\documentclass[11pt]{smfart}
\usepackage{color}
\usepackage{amssymb,verbatim}
\usepackage{hyperref}
\usepackage{amsthm,array,amssymb,amscd,amsfonts,amssymb,latexsym, url}
\usepackage{amsmath}
\usepackage[all]{xy}
\usepackage[french]{babel}

\newcounter{spec}
{\end{list}}

\renewcommand{\P}{{\mathbf P}}

\newcommand{\Z}{{\mathbb Z}}
\newcommand{\Q}{{\mathbb Q}}
\newcommand{\C}{{\mathbb C}}

\renewcommand{\L}{{\mathbb L}}
\newcommand{\oi}{\hskip1mm {\buildrel \simeq \over \rightarrow} \hskip1mm}

\renewcommand{\lim}{\varprojlim}

\def\lra{\longrightarrow}

\def\X{{{\overline X}}}
\def\FF {{\overline{\mathbb F}}}
\def\P{{\bf P}}

\numberwithin{equation}{section}

\newfont{\gothic}{eufb10}

\renewcommand{\qed}{{\hfill$\square$}}

\newtheorem{theo}{Th\'{e}or\`{e}me}[section]
\newtheorem{prop}[theo]{Proposition}

\newtheorem{lem}[theo]{Lemme}
\newtheorem{cor}[theo]{Corollaire}
\theoremstyle{definition}
\newtheorem{defi}[theo]{D\'efinition}
\theoremstyle{remark}
\newtheorem{rema}[theo]{Remarque}

\newcommand{\bthe}{\begin{theo}}
\newcommand{\ble}{\begin{lem}}
\newcommand{\bpr}{\begin{prop}}
\newcommand{\bco}{\begin{cor}}
\newcommand{\bde}{\begin{defi}}
\newcommand{\ethe}{\end{theo}}
\newcommand{\ele}{\end{lem}}
\newcommand{\epr}{\end{prop}}
\newcommand{\eco}{\end{cor}}
\newcommand{\ede}{\end{defi}}

\newcommand{\et}{{\operatorname{\acute{e}t}}}

\newcommand{\Pic}{\operatorname{Pic}}

\newcommand{\F}{{\mathbb F}}

\newcommand{\Y}{{\overline Y}}

\def\k{{\overline k}}
\DeclareFontFamily{U}{wncy}{}
\DeclareFontShape{U}{wncy}{m}{n}{%
<5>wncyr5%
<6>wncyr6%
<7>wncyr7%
<8>wncyr8%
<9>wncyr9%
<10>wncyr10%
<11>wncyr10%
<12>wncyr6%
<14>wncyr7%
<17>wncyr8%
<20>wncyr10%
<25>wncyr10}{}
\DeclareMathAlphabet{\cyr}{U}{wncy}{m}{n}

\begin{document}

  \title[Cohomologie non ramifi\'ee des hypersurfaces de Fano]
  {Troisi\`eme groupe de cohomologie non ramifi\'ee des hypersurfaces de Fano}

\author{J.-L. Colliot-Th\'el\`ene}
\address{CNRS,  Universit\'e Paris Sud et Paris-Saclay \\Math\'ematiques, B\^atiment 425\\91405 Orsay C\'edex\\France}
\email{jlct@math.u-psud.fr}

\date{submitted  August 3rd, 2017;  revised  October 15th,  2017}
\maketitle

 \begin{abstract}
 Sur un corps alg\'ebriquement clos  
 et sur un corps fini,
on \'etablit de nouveaux  r\'esultats d'annulation pour la cohomologie non ramifi\'ee de degr\'e~3
 des hypersurfaces de Fano.  \end{abstract}

\begin{altabstract}
We establish the vanishing  of degree three unramified cohomology for
several  new types of Fano hypersurfaces when the ground field is either finite or
algebraically closed of arbitrary characteristic.
\end{altabstract}

Soit $X$ une vari\'et\'e 
projective, lisse, g\'eom\'etriquement connexe
sur un  corps $k$
et $\ell \neq {\rm car}(k)$
un nombre premier.
Pour tout couple d'entiers $i \geq 0$ et $j \in \Z$,   le groupe  de cohomologie non ramifi\'ee
$H^i_{nr}(X,\Q_{\ell}/\Z_{\ell}(j))$ est par d\'efinition le  
 groupe des sections globales    du
faisceau  pour la topologie de Zariski sur $X$ associ\'e au pr\'efaisceau 
qui \`a  un ouvert $U \subset X$ associe le groupe de cohomologie \'etale  $H^i_{\et}(U,\Q_{\ell}/\Z_{\ell}(j))$
de $U$ \`a valeurs dans le groupe
des racines $\ell$-primaires de l'unit\'e tordues $j$  fois. Les propri\'et\'es g\'en\'erales
de ces groupes sont d\'ecrites dans le rapport \cite{CTBarbara}.
Le groupe $H^2_{\et}(U,\Q_{\ell}/\Z_{\ell}(1))$ est la composante $\ell$-primaire du groupe de Brauer de $X$.
Les groupes $H^i_{nr}(X,\Q_{\ell}/\Z_{\ell}(j))$ sont des invariants $k$-birationnels
des vari\'et\'es projectives et lisses. On a une  application naturelle  
 du groupe de cohomologie galoisienne $H^{i}(k,\Q_{\ell}/\Z_{\ell}(j)) =   H^{i}_{\et}(k, \Q_{\ell}/\Z_{\ell}(j))$
dans le groupe $H^i_{nr}(X,\Q_{\ell}/\Z_{\ell}(j))$, application
 qui est un isomorphisme si $X$ est $k$-birationnelle
\`a un espace projectif ${\P}^m_{k}$. 

On s'int\'eresse ici au  groupe $H^3_{nr}(X,\Q_{\ell}/\Z_{\ell}(2))$.
Ce groupe joue un r\^ole important dans l'\'etude \cite{CTV, K, CTK,  CT}
de l'application ``cycle'' sur le groupe de Chow des cycles de codimension 2
$$ {\rm cyc}_{X} : CH^2(X) \otimes \Z_{\ell} \to H^4_{\et}(X,\Z_{\ell}(2)).$$

Pour une hypersurface cubique lisse $X \subset \P^n_{\C}$ sur le corps des complexes,
$n=4$ et $n=5$,
on sait que l'on a  $H^3_{nr}(X,\Q_{\ell}/\Z_{\ell}(2))=0$ pour tout $\ell$.
C'est une cons\'equence \cite[Thm. 1.1]{CTV}  de la conjecture de Hodge enti\`ere
pour les cycles de codimension 2 sur ces hypersurfaces cubiques. Pour $n=4$, 
cette conjecture est facile \`a \'etablir (voir  le th\'eor\`eme \ref{cycleappalgclos} ci-dessous). C'est aussi un cas tr\`es particulier d'un th\'eor\`eme  g\'en\'eral de C. Voisin sur les solides unir\'egl\'es.
Pour $n=5$, cette conjecture  fut d\'emontr\'ee par C.~Voisin \cite[Thm. 18]{V}.

Dans  \cite[\S 5.3]{CT}, j'ai  discut\'e des extensions de ce r\'esultat 
aux hypersurfaces lisses de degr\'e $d \leq n$ dans un espace projectif $\P^n_{\C}$ 
avec $n$ quelconque. Par la formule d'adjonction, ce sont exactement les
hypersurfaces lisses de Fano, c'est-\`a-dire \`a fibr\'e anticanonique ample.

Dans cet article, on consid\`ere la situation sur un
corps  alg\'ebriquement clos de caract\'eristique quelconque,
et sur un corps fini.

Plus pr\'ecis\'ement, pour $X \subset \P^n_{k}$ une hypersurface lisse de degr\'e $d\leq n$
sur un corps $k$ de caract\'eristique  diff\'erente de $\ell$, 
on \'etablit    $$H^3_{nr}(X,\Q_{\ell}/\Z_{\ell}(2))=0$$
dans chacun des cas suivants :

(i) $k$ alg\'ebriquement clos et $n \neq 5$ (Th\'eor\`eme \ref{cycleappalgclos})

 (ii) $k=\F$ fini et $n \neq 4, 5$ (Th\'eor\`eme \ref{cubiquefini});

(iii) $k$ alg\'ebriquement clos (de caract\'eristique diff\'erente de 2 et 3),     
$d=3$ et $n=5$ (Th\'eor\`eme \ref{charlespirutka});

(iv) $k=\F$ fini, $d=3$  et $n=4$ (Th\'eor\`eme \ref{parisu}).

Le cas des hypersurfaces cubiques  lisses dans $\P^5_{\F}$ reste ouvert.

La d\'emonstration du cas (iii) repose sur un  th\'eor\`eme de Charles et Pirutka \cite{CP}. 
Dans le cas   (iv), on offre deux d\'emonstrations, utilisant toutes deux la th\'eorie
du corps de classes sup\'erieur de K. Kato et S. Saito. L'une de ces d\'emonstrations
passe par un  th\'eor\`eme de Parimala et Suresh \cite{PS}.

Pour $X$ une vari\'et\'e sur un corps $k$ et $\k$ une cl\^oture s\'eparable de $k$,
on note $\X=X \times_{k} \k$.

\section{Quelques rappels}

\begin{lem}\label{Fano}
Soit $\F$ un corps fini. 
Soit $X \subset \P^n_{\F}$, $n \geq 4$, une hypersurface cubique lisse.
Le pgcd des degr\'es des  extensions finies $L$ de $\F$ sur lesquelles $X_{L}$
poss\`ede une $L$-droite est \'egal \`a 1.
\end{lem}

\begin{proof} D'apr\`es Fano, Altman et Kleiman \cite{AK}, sur tout corps $k$,
la  vari\'et\'e de Fano $F=F(X)$ des droites de $X \subset \P^n_{k}$,  est non vide, projective et lisse \cite[Cor.  1.12]{AK}
pour $n \geq 3$ et g\'eom\'etriquement connexe pour $n \geq 4$ \cite[Thm 1.16 (i)]{AK}. 
Sur un corps fini $\F$, les estimations de Lang-Weil donnent le r\'esultat. \end{proof}
 
 \begin{rema}
 Des r\'esultats  pr\'ecis sur l'existence de droites sur le corps fini $\F$ lui-m\^eme sont obtenus dans \cite{DLR}.
 \end{rema}
 
 \begin{prop}\label{surfaces}
 Soit $X$ une surface projective,  lisse,
  g\'eom\'etriquement connexe 
  sur un corps $k$.
 Soit $\ell$ un nombre premier diff\'erent de la caract\'eristique de $k$.
 Si $k$ est alg\'ebriquement clos, ou si $k$ est fini,  $H^3_{nr}(X,\Q_{\ell}/\Z_{\ell}(2))=0$.
  \end{prop}
  \begin{proof}
  On a  $H^3_{nr}(X,\Q_{\ell}/\Z_{\ell}(2)) \subset H^3(k(X),\Q_{\ell}/\Z_{\ell}(2))$. Ce dernier groupe est
  nul si   $k$ est alg\'ebriquement clos, car la $\ell$-dimension cohomologique du corps des fonctions $k(X)$ est 2.

  Pour toute surface $X$ projective, lisse, 
  g\'eom\'etriquement
   connexe sur un corps fini
  et  $\ell$ premier diff\'erent de la caract\'eristique de $k$, on a $H^3_{nr}(X,\Q_{\ell}/\Z_{\ell}(2))=0$
(Sansuc, Soul\'e et l'auteur \cite[Rem. 2 p. 790]{CTSS}; K.~Kato   \cite[Thm. 0.7 and Corollary]{Kato}).

 \end{proof}

\begin{prop}\label{exptors}
Soit $n\geq 3$ un entier et soit $X \subset \P^n_{k}$ une hypersurface cubique lisse sur un corps $k$.
Soit $\ell$ un nombre premier diff\'erent de la caract\'eristique de $k$.

(i) Si $X$ poss\`ede un z\'ero-cycle de degr\'e 1,
le quotient du groupe $H^3_{nr}(X,\Q_{\ell}/\Z_{\ell}(2))$ par l'image de
$H^3(k,\Q_{\ell}/\Z_{\ell}(2))$ est annul\'e par 6. 

(ii) Si $X$ contient une droite $k$-rationnelle, le quotient du groupe
$H^3_{nr}(X,\Q_{\ell}/\Z_{\ell}(2))$ par l'image de
$H^3(k,\Q_{\ell}/\Z_{\ell}(2))$ est annul\'e par~2.

(iii) Si $k$ est alg\'ebriquement clos, $H^3_{nr}(X,\Q_{\ell}/\Z_{\ell}(2))$ est annul\'e par~2.

(iv) Si $k$ est fini, $H^3_{nr}(X,\Q_{\ell}/\Z_{\ell}(2))$ est annul\'e par~2.
\end{prop}

\begin{proof}
 Les \'enonc\'es \cite[Thm. 1.4]{ACTP} et  \cite[Prop. 2.1]{ACTP} donnent que ce quotient est annul\'e  par 6 si $X$ poss\`ede un z\'ero-cycle de degr\'e 1,
  et par 2 si $X$ contient une droite $k$-rationnelle.
Ceci \'etablit (i), (ii) et (iii).  Pour $k$ un corps fini, $X$ poss\`ede un z\'ero-cycle de degr\'e 1, et m\^eme un point rationnel. L'\'enonc\'e (iv) pour $n=3$ est un cas particulier de la proposition \ref{surfaces}.
 Pour $n\geq 4$, l'\'enonc\'e (iv) r\'esulte de la combinaison de l'\'enonc\'e (ii),  du lemme  \ref{Fano}  et d'un argument de corestriction-restriction.
\end{proof}
 
\section{Hypersurfaces de Fano dans $\P^n_{k}$, $k$ alg\'ebriquement clos, $n \neq 5$}
On \'etend en toute caract\'eristique des r\'esultats de \cite{CT}.
On en profite pour rectifier la d\'emonstration de \cite[Thm. 5.6 (vi)]{CT} pour une hypersurface dans $\P^4$.
 
\begin{theo}\label{cycleappalgclos}
Soit $n\geq 3$ un entier, et soit $X \subset \P^n_{k}$ une hypersurface lisse de degr\'e $d$
sur un corps  alg\'ebriquement clos $k$.
Soit $\ell$ un nombre premier diff\'erent de la caract\'eristique de $k$. 
 
(i) Pour $n=3$ et $n\geq 6$,  l'application cycle $$ {\rm cyc}_{X} :  CH^2(X)\otimes \Z_{\ell} \to  H^4(X,\Z_{\ell}(2))$$
est surjective.

(ii) Pour $n=4$ et $d\leq 4$, l'application cycle $$ {\rm cyc}_{X} :  CH^2(X)\otimes \Z_{\ell} \to  H^4(X,\Z_{\ell}(2))$$
est surjective.

(iii) Pour $n \neq 5$ et $d \leq n$,
 on a  $H^3_{nr}(X,\Q_{\ell}/\Z_{\ell}(2))=0$.

\end{theo}

\begin{proof}

\'Etablissons (i).
 Pour $n=3$, la classe de tout $k$-point de $X$ engendre le $\Z_{\ell}$-module 
 $ H^4(X,\Z_{\ell}(2)) \simeq \Z_{\ell}$. L'\'enonc\'e (i) est donc clair pour $n=3$.

 Supposons  $n \geq 4$.
Soit $U =\P^n_{k} \setminus X$. Pour tout entier $m>0$,   on a  la suite exacte de cohomologie 
\'etale \`a supports propres   \cite[III.1.30]{M} :
$$H^4_{c}(U,\Z/{\ell}^m(2)) \to  H^4(\P^n, \Z/{\ell}^m(2)) \to H^4(X, \Z/{\ell}^m(2)) \to H^5_{c}(U,\Z/{\ell}^m(2)).$$
Les groupes finis $H^{i}_{c}(U,\Z/{\ell}^m(2))$
et $H^{2n-i}(U,\Z/{\ell}^m(2n-2))$ sont duaux (dualit\'e de Poincar\'e \cite[VI.11.2]{M}).

 Pour $n\geq 6$,
on a $2n-4 >2n-5> n$. Le th\'eor\`eme de Lefschetz affine \cite[VI.7.2]{M}
donne   
$H^{2n-4}(U,\Z/{\ell}^m(2n-2))=0$
et $H^{2n-5}(U,\Z/{\ell}^m(2n-2))=0$.

La fl\`eche de restriction $ H^4(\P^n, \Z/{\ell}^m(2)) \to H^4(X, \Z/{\ell}^m(2)) $
est donc un isomorphisme de groupes finis pour tout $m$.
La fl\`eche de restriction 
$$\Z_{\ell}= H^4(\P^n, \Z_{\ell}(2)) \to  H^4(X,\Z_{\ell}(2))$$
est donc un isomorphisme.
 Ceci implique 
 que l'application cycle
$${\rm cyc}_{X} :  CH^2(X) \otimes \Z_{\ell} \to H^4(X,\Z_{\ell}(2)) $$
est surjective. Ceci \'etablit (i) pour $n \geq 6$.

\smallskip 

Pour $n>4$, la consid\'eration de la suite exacte
$$  H^3(\P^n, \Z/{\ell}^m(2)) \to H^3(X, \Z/{\ell}^m(2)) \to H^4_{c}(U,\Z/{\ell}^m(2)),$$
la dualit\'e de Poincar\'e et le th\'eor\`eme de Lefschetz affine donnent    alors  $H^3(X, \Z/{\ell}^m(2))=0$ pour tout $m$ et donc $H^3(X,\Z_{\ell}(2))=0$. Ceci sera utilis\'e dans la d\'emonstration du th\'eor\`eme \ref{cubiquefini}
ci-apr\`es.
 
 \smallskip
 
 \'Etablissons  l'\'enonc\'e (ii).
 Soit donc $n=4$.  
 L'argument qui suit corrige celui  donn\'e  dans \cite[Thm. 5.6 (vi)]{CT}.

 Pour tout  degr\'e $d$, et tout entier $m>0$,    la fl\`eche de restriction $H^2(\P^4,\Z/\ell^m) \to H^2(X, \Z/\ell^m)$
 est un isomorphisme, comme on voit en utilisant la suite exacte de cohomologie \'etale \`a supports,
 la dualit\'e de Poincar\'e sur $U=\P^4\setminus X$, et le th\'eor\`eme de Lefschetz affine. 
 Ceci implique  $H^2(X, \Z/\ell^m) \simeq \Z/\ell^m$, et ceci implique que l'application
 cycle $CH^1(X)/\ell^m   \to H^2(X,\mu_{\ell^m})$ d\'efinie via l'application de Kummer
 $\Pic(X)/\ell^m   \to H^2(X,\mu_{\ell^m})$
  est un isomorphisme.


Le cup-produit sur la cohomologie \'etale
$$H^4(X,\Z/\ell^m(2))    \times    H^2(X,\Z/\ell^m(1))  \to  H^6(X, \Z/\ell^m(3))=\Z/\ell^m$$
est un accouplement non d\'eg\'en\'er\'e de groupes finis (dualit\'e de Poincar\'e).
D'apr\`es ce qui pr\'ec\`ede,  chacun des deux termes de cet accouplement est isomorphe
\`a $\Z/\ell^m$.  Consid\'erons le diagramme :
 $$ \begin{array}{ccccccccc}
CH^2(X)/\ell^m   &  \times &  CH^1(X)/\ell^m  & \lra &   \Z/\ell^m \\
 \downarrow                       &                                                            &     \downarrow{\simeq}             &                      &\downarrow{=}\\
H^4(X,\Z/\ell^m(2))  &  \times  &  H^2(X,\Z/\ell^m(1))    & \lra &   \Z/\ell^m, \\
\end{array}$$
o\`u l'accouplement sup\'erieur est donn\'e par l'intersection des cycles.
 
  Pour $\ell \neq 2= ({\rm dim}(X)-1) !$, ce diagramme est commutatif \cite[Prop VI.10.7]{M},
Pour   tout premier $\ell$, il commute sur les couples de cycles $(Z_{1}, Z_{2})$ transverses l'un \`a l'autre
 \cite[Prop. VI.9.5]{M}.
Soit $Y=H\cap X \subset X$ la trace d'un hyperplan $H \subset \P^4$.
 Sous l'hypoth\`ese   $d \leq 4$, l'hypersurface $X$ contient une droite $L \subset \P^4$.
Ceci est  bien connu pour $d=3$; pour un \'enonc\'e g\'en\'eral, voir \cite[Thm. 2.1]{D}.
Dans l'accouplement sup\'erieur, on a $(L, Y)=1$.  
En appliquant  \cite[Prop. VI.9.5]{M}, on voit
 que la classe de  cycle de $L$ dans  $H^4(X,\Z/\ell^m(2)) \simeq \Z/\ell^m$ engendre
ce groupe.
Ainsi l'application
 cycle
 $$ {\rm cyc}_{X} :   CH^2(X)\otimes \Z_{\ell} \to  H^4(X,\Z_{\ell}(2))$$
est  surjective. Ceci \'etablit (ii) pour $n=4$.

\smallskip

Montrons maintenant (iii).
D'apr\`es \cite[Thm. 1.1]{K}  ou  \cite[Thm. 2.2]{CTK},  la surjectivit\'e de
  $${\rm cyc}_{X} :  CH^2(X) \otimes \Z_{\ell} \to H^4(X,\Z_{\ell}(2))=\Z_{\ell}$$
 implique que le groupe
 $H^3_{nr}(X,\Q_{\ell}/\Z_{\ell}(2))$ est divisible.
 
D'apr\`es un th\'eor\`eme de Roitman  (\cite{R}, voir aussi \cite[\S 4]{CL}),
  l'hypoth\`ese $d \leq n$ implique
    que sur tout corps alg\'ebriquement clos $K$ contenant $k$,
 l'application degr\'e $CH_{0}(X_{K}) \to \Z$ sur le groupe de Chow des
 z\'ero-cycles est un isomorphisme. 
 D'apr\`es un argument g\'en\'eral (voir \cite[Prop. 3.2]{CTK}),
ceci  implique l'existence d'un entier $N>0$
 qui annule $H^3_{nr}(X,\Q_{\ell}/\Z_{\ell}(2))$.  

 Sous l'hypoth\`ese $n\neq 5$ et $d \leq n$, on a donc \'etabli que le
groupe $H^3_{nr}(X,\Q_{\ell}/\Z_{\ell}(2))$ est  divisible et d'exposant fini. Il est
 donc nul.  
 \end{proof}
 \begin{rema}
  Pour $k=\C$ et $X \subset \P^n_{k}$ comme ci-dessus avec $d \leq n$  et tout corps $F$ contenant $k$,
 et  pour $n \geq 6$, on a \'etabli dans \cite[Thm. 5.6 (vii)]{CT} que la fl\`eche naturelle
$$H^3(F, \Q_{\ell}/\Z_{\ell}(2)) \to H^3_{nr}(X_{F},\Q_{\ell}/\Z_{\ell}(2))$$
est un isomorphisme. Il est tr\`es vraisemblable que ce r\'esultat vaut 
sur tout corps $k$ alg\'ebriquement clos, 
avec $\ell$ distinct de la caract\'eristique de $k$.
   \end{rema}

\section{Hypersurfaces de Fano  dans $\P^n_{\F}$, $\F$ fini, $n=3$ et $n\geq 6$}
 
\begin{theo}\label{cubiquefini}
Soit $n\geq 3$ un entier et soit $X \subset \P^n_{\F}$ une hypersurface   lisse de degr\'e $d \leq n$ sur un corps  fini  $\F$.
Soit $\ell$ un nombre premier diff\'erent de la caract\'eristique de $\F$. Pour  $n=3$ et pour $n\geq 6$,  on a  $H^3_{nr}(X,\Q_{\ell}/\Z_{\ell}(2))=0$.  \end{theo}

\begin{proof}
D'apr\`es la proposition  \ref{surfaces}, on peut supposer $n\geq 6$.

Pour $n \geq 6 $, on a \'etabli dans la d\'emonstration du th\'eor\`eme \ref{cycleappalgclos} que l'on a
  $H^3(\X,\Z_{\ell}(2))=0$ et que la restriction
  $$\Z_{\ell}=H^4(\P^n_{\overline{\F}},\Z_{\ell}(2)) \to  H^4(\X,\Z_{\ell}(2))$$
  est un isomorphisme. 
  Pour toute $\F$-vari\'et\'e $Y$, on dispose de la suite exacte d\'eduite de la suite spectrale de Leray
$$0 \to H^1(\F,H^3(\overline{Y},\Z_{\ell}(2))) \to H^4(Y,\Z_{\ell}(2)) \to H^0(\F,H^4(\overline{Y},\Z_{\ell}(2))) \to 0.$$
La comparaison de cette suite pour $Y=\P^n_{\F}$ et pour $Y=X$
donne que l'application cycle
$${\rm cyc}_{X} :  CH^2(X) \otimes \Z_{\ell} \to H^4(X,\Z_{\ell}(2))=\Z_{\ell}$$
est surjective.

D'apr\`es \cite[Thm. 1.1]{K}  ou  \cite[Thm. 2.2]{CTK}, sur un corps fini $\F$,  la surjectivit\'e de
  $${\rm cyc}_{X} :  CH^2(X) \otimes \Z_{\ell} \to H^4(X,\Z_{\ell}(2))=\Z_{\ell}$$
 implique que le groupe
 $H^3_{nr}(X,\Q_{\ell}/\Z_{\ell}(2))$ est divisible. 

 Comme rappel\'e dans la d\'emonstration du th\'eor\`eme \ref{cycleappalgclos}, l'hypoth\`ese $d \leq n$,
le th\'eor\`eme de Roitman \cite{R} et l'argument donn\'e dans
 \cite[Prop. 3.2]{CTK} impliquent que $H^3_{nr}(X,\Q_{\ell}/\Z_{\ell}(2))$ est
 d'exposant fini.
 
 Le groupe  $H^3_{nr}(X,\Q_{\ell}/\Z_{\ell}(2))$ est divisible et d'exposant fini,
 il est donc nul.
 \end{proof}

\section{Hypersurfaces cubiques dans $\P^5_{k}$, $k$ alg\'ebriquement clos}

D\'ej\`a pour les hypersurfaces cubiques, le th\'eor\`eme \ref{cycleappalgclos}, sur un corps alg\'ebriquement clos,
 laisse ouvert le cas $n=5$.
Pour  $X \subset \P^5_{\C}$ une hypersurface cubique lisse sur le corps des complexes,
  Claire Voisin \cite[Thm. 18]{V} a \'etabli 
la conjecture de Hodge enti\`ere dans ce contexte.
 D'apr\`es  \cite{CTV},   ceci implique  $H^3_{nr}(X,\Q/\Z(2)) = 0$,
 et \cite[Thm. 4.4.1]{CTBarbara} montre alors que le r\'esultat vaut pour
 toute hypersurface cubique lisse $X \subset \P^5_k$ 
 sur un corps~$k$ alg\'ebriquement clos
 de caract\'eristique z\'ero.

En utilisant le travail de Charles et Pirutka \cite{CP}, 
on obtient l'analogue de ce r\'esultat sur tout corps alg\'ebriquement clos, 
avec une restriction mineure sur la caract\'eristique.
 
\bthe\label{charlespirutka}
  Soit $k$ un corps alg\'ebriquement clos de caract\'eristique diff\'erente de 2 et 3.
Soit $X \subset \P^5_k$ une hypersurface cubique lisse.
Soit $\ell$ premier diff\'erent de la caract\'eristique de $k$.
On a  $H^3_{nr}(X,\Q_{\ell}/\Z_{\ell}(2))=0$.
\ethe
 
\begin{proof}
D'apr\`es la proposition \ref{exptors},
le groupe $H^3_{nr}(X,\Q_{\ell}/\Z_{\ell}(2))$
est d'exposant fini, en fait divisant 2.
La proposition \ref{exptors} donne donc d\'ej\`a le r\'esultat pour $\ell \neq 2$.

Par une variante du lemme de rigidit\'e de Suslin \cite[Thm. 4.4.1]{CTBarbara},
pour \'etablir ce dernier \'enonc\'e $H^3_{nr}(X,\Q_{\ell}/\Z_{\ell}(2))=0$, on peut se limiter \`a consid\'erer le cas o\`u $k$
est une cl\^{o}ture alg\'ebrique d'un corps $F$ de type fini sur le corps premier,
et o\`u $X=X_{0}\times_{F}k$ pour $X_{0} \subset \P^5_{F}$ une  hypersurface cubique lisse.

On consid\`ere l'application cycle
$CH^2(X) \otimes \Z_{\ell} \to H^4(X,\Z_{\ell}(2)).$
Elle respecte l'action du groupe de Galois ${\rm Gal}(k/F)$.
Elle envoie donc  le groupe des cycles dans le sous-groupe
$$ H^4(X,\Z_{\ell}(2))^{f} \subset H^4(X,\Z_{\ell}(2))$$
des classes dont le stabilisateur est un sous-groupe ouvert.

Comme $H^4(X,\Z_{\ell}(2))$ est un $\Z_{\ell}$-module de type fini
et l'action de  ${\rm Gal}(k/F)$ est continue, 
le conoyau de
$$ H^4(X,\Z_{\ell}(2))^{f} \to H^4(X,\Z_{\ell}(2))$$
est un groupe sans torsion \cite[Lemme 4.1]{CTK}.

Charles et Pirutka  \cite[Thm. 1.1]{CP}  ont montr\'e que l'application
$$CH^2(X) \otimes \Z_{\ell}  \to H^4(X,\Z_{\ell}(2))^{f}$$
est surjective. On conclut que le conoyau de
$${\rm cyc}_{X} :  CH^2(X) \otimes \Z_{\ell} \to H^4(X,\Z_{\ell}(2))$$
est un groupe sans torsion.
D'apr\`es   \cite[Thm. 1.1]{K}   ou \cite[Thm. 2.2]{CTK}, 
  le groupe fini donn\'e par la torsion du
conoyau de l'application cycle
$${\rm cyc}_{X} :  CH^2(X) \otimes \Z_{\ell} \to H^4(X,\Z_{\ell}(2))$$
co\"{\i}ncide avec le groupe quotient de 
$H^3_{nr}(X,\Q_{\ell}/\Z_{\ell}(2))$ par son sous-groupe divisible maximal.
D'apr\`es la proposition \ref{exptors}, le groupe
$H^3_{nr}(X,\Q_{\ell}/\Z_{\ell}(2))$ est d'exposant fini.
Ceci \'etablit
$H^3_{nr}(X,\Q_{\ell}/\Z_{\ell}(2)) = 0.$
\end{proof}

 \begin{rema}
Pour     $X\subset \P^5_{\C}$ une hypersurface cubique lisse, la conjecture de Hodge rationnelle
(\`a coefficients dans $\Q$)
pour les cycles de codimension deux
est connue  depuis 1977  (Zucker \cite{Z}, Murre \cite{murre}).
 La nullit\'e de $H^3_{nr}(X,\Q_{\ell}/\Z_{\ell}(2))$
\'etablie ci-dessus
et \cite[Thm. 1.1]{CTV} redonnent donc la conjecture de Hodge enti\`ere pour
les cycles de codimension deux sur ces hypersurfaces, c'est-\`a-dire le r\'esultat \'etabli en 2007 par
C.~Voisin   \cite[Thm. 18]{V} \cite[Thm. 3.11]{V2}.  Il convient cependant d'observer que la d\'emonstration ci-dessus repose de fa\c con essentielle sur 
\cite{CP}, dont les
m\'ethodes g\'eom\'etriques sont inspir\'ees de celles de \cite{V}
(qui cite \cite{Z}).
\end{rema}

\section{Hypersurfaces cubiques dans $\P^4_{\F}$, $\F$ corps fini}

Pour les hypersurfaces cubiques lisses sur un corps fini, le travail \cite{PS} de Parimala et Suresh   permet de compl\'eter le th\'eor\`eme  \ref{cubiquefini} pour $n=4$.

\bthe\label{parisu}
{\it Soit $X \subset \P^4_{\F}$ une hypersurface cubique lisse sur un corps fini $\F$
de caract\'eristique diff\'erente de $2$.

(i) Pour tout $\ell$ premier diff\'erent de la caract\'eristique de $\F$, on a
$H^3_{nr}(X,\Q_{\ell}/\Z_{\ell}(2))=0$.

(ii) Soit $\FF$ une cl\^oture alg\'ebrique de $\F$ et $G={\rm Gal}(\FF/\F)$.
L'application naturelle 
$$ CH^2(X) \to CH^2(\X)^G$$ est un isomorphisme.

(iii) L'application cycle
$${\rm cyc}_{X} :  CH^2(X) \otimes \Z_{\ell} \to H^4(X,\Z_{\ell}(2))$$
est surjective.}
\ethe

 \begin{proof}
(i)  Le cas $\ell \neq 2$ r\'esulte d\'ej\`a de la proposition \ref{exptors}.
Pour d\'emontrer la proposition, par le lemme \ref{Fano} et un argument de  restriction-corestriction,
on peut supposer que $X$ contient une droite $L \subset X$ d\'efinie sur le corps $\F$.
En \'eclatant $X$ le long de $L$,
on trouve une $\F$-vari\'et\'e projective et lisse~$Y$ $\F$-birationnelle \`a $X$ et munie d'une structure de fibration en coniques sur $\P^2_{\F}$.
Le th\'eor\`eme de Parimala-Suresh \cite[Cor. 5.6]{PS}  
 donne alors $H^3_{nr}(Y,\Q_{\ell}/\Z_{\ell}(2))=0$, et donc
$H^3_{nr}(X,\Q_{\ell}/\Z_{\ell}(2))=0$.

(ii) On sait (th\'eor\`eme de Lefschetz faible) que $H^3(\X, \Z_{\ell})$ est sans torsion. 
La nullit\'e de $H^3_{nr}(X,\Q_{\ell}/\Z_{\ell}(2))$ et 
\cite[Cor. 6.9]{CTK} donnent (ii).

(iii)  Comme  $X$ est g\'eom\'etriquement unirationnelle de dimension 3, le 
conoyau de l'application cycle
$CH^2(X) \otimes \Z_{\ell} \to H^4(X,\Z_{\ell}(2))$
est un groupe fini \cite[Prop. 3.23]{CTK}.
   D'apr\`es \cite[Thm. 1.1]{K}  ou  \cite[Thm. 2.2]{CTK},  la torsion du  conoyau  de l'application cycle s'identifie
au quotient de $H^3_{nr}(X,\Q_{\ell}/\Z_{\ell}(2))$ par son sous-groupe
divisible maximal. De (i) r\'esulte donc (iii).
\end{proof}

\begin{rema}
La d\'emonstration du 
 th\'eor\`eme de Parimala et Suresh \cite{PS} utilise
  un r\'esultat de th\'eorie du corps de classes sup\'erieur, \`a savoir 
 la nullit\'e de $H^3_{nr}(S, \Q_{\ell}/\Z_{\ell}(2))$
pour $S$ une surface projective et lisse sur un corps fini (\cite[Rem. 2, p. 790]{CTSS}; \cite[Thm. 0.7, Cor]{Kato}).
Elle utilise aussi beaucoup d'autres arguments d\'elicats.

En utilisant la th\'eorie du corps de classes sup\'erieur, et le lien entre   la surface de Fano
des droites de $X$ et le groupe des cycles de codimension 2 de $X$,
on peut donner une  d\'emonstration alternative   du th\'eor\`eme
\ref{parisu}. Soit $Y/\F$ la surface de Fano de $X$, qui param\'etrise les droites de $X$. 
C'est une surface projective, lisse, g\'eom\'etriquement connexe \cite[Cor. 1.12]{AK}, qui poss\`ede donc un z\'ero-cycle de degr\'e~1 sur le corps fini $\F$.

 La famille universelle des droites de $X$
d\'efinit une correspondance entre $Y$ et $X$ qui induit un homomorphisme
$CH_{0}(Y) \to CH^2(X)$, lequel induit une application $A_{0}(Y) \to CH^2_{0}(X)$,
o\`u l'on a not\'e $A_{0}(Y) \subset CH_{0}(Y)$ le sous-groupe des z\'ero-cycles de degr\'e z\'ero, et $CH^2_{0}(X) \subset CH^2(X)$ le sous-groupe des 1-cycles d'intersection nulle
avec une section hyperplane. Sur un corps de caract\'eristique diff\'erente de 2,
on sait  \cite[VI,VII]{murrecubic}
que  l'application
 $A_{0}(\Y) \to CH^2_{0}(\X)$ se factorise comme
 $$ A_{0}(\Y) \to Alb_{Y}(\FF) \oi CH^2_{0}(\X).$$
 D'apr\`es le th\'eor\`eme de Roitman,   l'application d'Albanese
 $ A_{0}(\Y) \to Alb_{Y}(\FF),$
qui est surjective, a son noyau uniquement divisible (en fait, pour $\F$ corps
fini, cette fl\`eche est un isomorphisme). Ceci assure que l'application
$A_{0}(\Y)^G \to CH^2_{0}(\X)^G$ est surjective. On a le diagramme commutatif
\[\xymatrix{
A_{0}(Y)  \ar[d] \ar[r] & CH^2_{0}(X)  \ar[d]  \\
A_{0}(\Y)^G \ar[r] &  CH^2_{0}(\X)^G .
}\]
 La th\'eorie du corps de classes sup\'erieur
(Kato-Saito \cite[Prop. 9.1]{KS}) montre que, pour toute vari\'et\'e projective lisse  $Y$ g\'eom\'etriquement connexe
sur un corps fini, l'application $A_{0}(Y)\to A_{0}(\Y)^G$ est surjective
(pour $Y/\F$ une surface, voir aussi \cite[\S 6.2]{CTK}).
On conclut donc que $ CH^2_{0}(X) \to CH^2_{0}(\X)^G$ est surjectif, puis que
$CH^2(X) \to CH^2(\X)^G$ est surjectif. Ceci donne l'\'enonc\'e (ii) du th\'eor\`eme \ref{parisu}.
Comme on a $H^3_{nr}(\X, \Q_{\ell}/\Z_{\ell}(2))=0$,
l'\'enonc\'e (i) r\'esulte alors de (ii) et de
\cite[Cor. 6.9]{CTK}. L'application  $CH^2(X) \otimes \Z_{\ell} \to H^4(X,\Z_{\ell}(2))$
a son conoyau fini.
D'apr\`es \cite{K} ou \cite[Thm. 2.2]{CTK}, ce conoyau s'identifie
au quotient de $H^3_{nr}(X,\Q_{\ell}/\Z_{\ell}(2)) $ par son sous-groupe divisible maximal.
Ainsi l'application $CH^2(X) \otimes \Z_{\ell} \to H^4(X,\Z_{\ell}(2))$
est surjective. 
 \end{rema}

\begin{rema}
Sur un corps fini $\F$ et  pour un nombre premier $\ell \neq {\rm car}(\F)$, la question si l'on a  $H^3_{nr}(X,\Q_{\ell}/\Z_{\ell}(2))=0$ pour 
une hypersurface cubique lisse  $X \subset \P^5_{\F}$ reste ouverte dans le cas crucial $\ell=2$ (pour $\ell\neq 2$,  voir la Proposition \ref{exptors} (iv)).
Elle est \'equivalente \`a la question de la surjectivit\'e de l'application cycle
$$ {\rm cyc}_{X} :  CH^2(X)\otimes \Z_{\ell} \to H^4(X, \Z_{\ell}(2)).$$
\end{rema}

\end{document}